\documentclass[10pt,amscd]{amsart}
\usepackage{latexsym}
\usepackage{amsmath}
\usepackage{amssymb}
\usepackage[all]{xy}
\usepackage[dvips]{graphicx, color}

\newtheorem{thm}{Theorem}[section]

\newtheorem{prop}[thm]{Proposition}
\newtheorem{cor}[thm]{Corollary}
\newtheorem{lem}[thm]{Lemma}

\theoremstyle{definition}
\newtheorem{defn}[thm]{Definition}

\newtheorem{rem}[thm]{Remark}
\newtheorem{ex}[thm]{\bf Example}
\newtheorem{qu}[thm]{Question}
\newtheorem{pr}[thm]{Problem}
\newcommand{\Z}{{\mathbb Z}}

\newcommand{\K}{{\bold k}}

\newcommand{\Q}{{\mathbb Q}}

\newcommand{\C}{{\mathbb C}}
\newcommand{\mapright}[1]{%
 \smash{\mathop{%
  \hbox to 1cm{\rightarrowfill}}\limits_{#1}}}
\newcommand{\maprightd}[2]{%
 \smash{\mathop{%
  \hbox to 1.2cm{\rightarrowfill}}\limits^{#1}\limits_{#2}}}
\newcommand{\mapleft}[1]{%
 \smash{\mathop{%
  \hbox to 1cm{\leftarrowfill}}\limits_{#1}}}
\newcommand{\mapleftu}[1]{%
 \smash{\mathop{%
  \hbox to 1cm{\leftarrowfill}}\limits^{#1}}}
\newcommand{\mapleftud}[2]{%
 \smash{\mathop{%
  \hbox to 1.2cm{\leftarrowfill}}\limits^{#1}\limits_{#2}}}

\begin{document}
\title{C-symplectic poset structure on  a simply connected
 space} 

\author{Kazuya Hamada, Toshihiro Yamaguchi and Shoji Yokura }
\date{}
\address{Kinkowan High School, 4047, Hirakawa-cho,  Kagoshima, 891-0133, JAPAN}
\email{Hamada-kazuya@edu.pref.kagoshima.jp }

\address{Faculty of Education, Kochi University, 2-5-1, Akebono-cho, Kochi,780-8520, JAPAN}
\email{tyamag@kochi-u.ac.jp}

\address{Faculty of Science, Kagoshima University, 21-35, Korimoto 1-chome, Kagoshima,  890-0065, JAPAN}
\email{yokura@sci.kagoshima-u.ac.jp}
\maketitle

\begin{abstract} For a field $\K$ of characteristic zero,
we introduce a cohomologically symplectic poset structure ${\mathcal P}_{\K}(X)$ on  a simply connected space $X$
from the viewpoint of $\K$-homotopy theory.
It is given by the poset of inclusions of subgroups preserving c-symplectic structures
  in the group ${\mathcal E}(X_{\K} )$ of  $\K$-homotopy classes  of  $\K$-homotopy self-equivalences
of $X$,
 which is defined by the $\K$-Sullivan model of $X$. 
We observe that the height of the Hasse diagram of ${\mathcal P}_{\K}(X)$ added by  $1$,
  denoted by  c-s-${\rm depth}_{\K}(X)$,
 is finite 
and often depends on the field $\K$.
 In this paper, we will give some examples of ${\mathcal P}_{\K}(X)$.
\end{abstract}

\section{Introduction}

Compact K\"ahler manifolds are well-known
 to be  symplectic manifolds and they are also formal spaces from the viewpoint of minimal models (\cite{DGMS}). 
However, symplecitic manifolds are not necessarily formal spaces (e.g.\cite{FHT}).
Dennis Sullivan posed the following problem \cite{Su}(cf.\cite[Problem 4]{TO},\cite[Question 4.99]{FOT}):
 {\it Are there algebraic conditions on the minimal model
of a compact manifold, implying the existence of a symplectic structure on it ?}
 There are interesting approaches to answer this question  
for certain cases (e.g. see \cite{LO1}, \cite{LO}, \cite{TO}). 
In this paper we discuss  simply connected
{cohomologically symplectic structures}, instead of the usual symplectic structures.
A Poincar\'{e} duality space (cf. \cite[Definition 3.1]{FOT}) $Y$ is said to have a {\it cohomologically symplectic} (abbr., {\it c-symplectic}) structure
if there is  a  rational cohomology 
class $\omega\in H^2(Y;\Q)$ such that $\omega^n$ is a top 
class for $Y$
(cf. \cite[Definition 4.87]{FOT}). 
 Since  $H^2(Y;\Z )\cong [Y,K(\Z , 2)]$ and $H^2(Y;\Q )\cong H^2(Y;\Z )\otimes \Q$,
we can consider the following situation:
\begin{enumerate}
\item an integral cohomology class $\mu\in H^2(Y;\Z)$ corresponds to the homotopy class of a map
$\mu: Y\to K(\Z,2)$ and 
\item $\mu\otimes 1_{\Q}$ in $ H^2(Y;\Z)\otimes \Q$
is a c-symplectic class on $Y$.
\end{enumerate}
Furthermore it follows from Serre's result that in the homotopy category
any map is ``the same" as a fibration $p:Y' \to K(\Z, 2).$ 
Thus from the beginning we can assume that $Y$ is the total space of a fibration to $K(\Z, 2)$.
Let $X$ be the fibre of the  fibration $\mu: Y\to K(\Z,2)$.
Then we rationalize this  fibration and study it by minimal models.

In \cite{SY}, J. Sato and the second named  author treated 
the above $c$-symplectic analogue of Sullivan's question,
 considering the fiber space $X$ of the fibration 
$\mu$. 
They asked the following question under the additional condition that the fiber space $X$ is simply connected:
 {\it 
What conditions on $X$ induce
a 
$c$-symplectic structure of the total space $Y_{\mu}$ of the fibration
$$\mu :X \to  Y_{\mu}\overset{p}{\to} K(\Z ,2).\ \ \ \ \ \ \ \ \ \ \ (*)$$}
\noindent
They call such a space $X$ \emph{pre-c-symplectic} (i.e. pre-cohomologically symplectic). 
For example they gave necessary and sufficient conditions for that the product $S^{k_1} \times S^{k_2} \times \cdots \times S^{k_n}$ of odd-dimensional spheres is pre-c-symplectic. 
In this paper, we 
furthermore consider the following
\begin{pr} {\it For a simply connected space $X$, how can  we classify c-symplectic   spaces $Y_{\mu}$ in 
 $(*)$  ?}
 \end{pr}
 When $S^1$ acts on $X$,
 we have a Borel fibration $(*)$  with $K(\Z , 2)=BS^1$.
In our case, it  is always a Borel fibration of a  free $S^1$-action
on a finite $CW$-complex $X'$ in the rational homotopy type of $X$ \cite[Proposition 4.2]{H}.
 In \cite{P}, 
 V. Puppe gives a classification of $S^1$-fibrations on a certain space $X$ having fixed points 
 via 
the  algebraic deformation theory of M. Gerstenhaber
over an algebraically closed field $\K$.
This classification gives a Hasse diagram (e.g. see \cite[page 350]{P}).
However, in the case of almost free $S^1$-actions, which have empty fixed points,
we cannot give such a classification using  the  algebraic deformation theory of Gerstenhaber.
Thus we cannot get such a Hasse diagram. 
In this paper, 
we can get a classification  of the total spaces $Y_{\mu}$
by using the groups of 
$\K$-homotopy classes of $\K$-homotopy self-equivalences,
by which we can get a poset denoted as ${\mathcal P}_{\K}(X)$ via the inclusions of the subgroups
 over a field of characteristic zero $\K$.
See  \cite{Y2} for some examples of such a poset of free $S^1$-actions on same rational types with $X$
when $\K=\Q$. Refer to \cite{F} for the other  general  classification 
given by  a  $\K$-deformation of Sullivan model.

In \S 2 
we recall $\K$-Sullivan models according to \cite{BM}.
In \S 3 
we define the \emph{``c-symplectic poset structure"} ${\mathcal P}_{\K}(X)$ for $X$ over $\K$ 
and a $\K$-homotopical  invariant,
the {\it c-symplectic depth of $X$ over $\K$},  denoted by   c-s-${\rm depth}_{\K}(X)$
which is defined 
by adding  $1$ to the height or the length of the Hasse diagram of the poset ${\mathcal P}_{\K}(X)$.
We give some results in \S 3 and their proofs are given in \S 4.
In \S 5 
we 
treat ${\mathcal P}_{\K}(X)$ mainly  in the   case that spaces
$X$ have the rational homotopy types of   products of  odd-spheres.

\section{ $\K$-Sullivan models } 

Let  $X_{\Q}$ and $f_{\Q}$ be    the rationalizations  \cite{HMR} of a simply connected 
space $X$ with finite rational cohomology
and a map $f$, respectively.
Let $M(X)=(\Lambda {V},d)$ be the Sullivan model of $X$.
  It is a freely generated  $\Q$-commutative differential graded algebra ($\Q$-DGA)
 with a $\Q$-graded vector space $V=\bigoplus_{i\geq 2}V^i$
 where $\dim V^i<\infty$ and a decomposable differential; i.e., 
$d(V^i) \subset (\Lambda^+{V} \cdot \Lambda^+{V})^{i+1}$ and $d \circ d=0$.
 Here  $\Lambda^+{V}$ is 
 the ideal of $\Lambda{V}$ generated by elements of positive degree. 
Denote the degree of a homogeneous element $x$ of a graded algebra 
by $|{x}|$.
Then  $xy=(-1)^{|{x}||{y}|}yx$ and $d(xy)=d(x)y+(-1)^{|{x}|}xd(y)$. 
Note that  $M(X)$ determines the rational homotopy type of $X$.
In particular,  $H^*(\Lambda {V},d)\cong H^*(X;\Q )$
and $V^i\cong Hom(\pi_i(X),\Q)$.
Refer to  \cite{FHT} for details.

Let $\K$ be a field of characteristic zero (then $\Q \subset \K$).
The $\K$-minimal model of a space $X$ is $ M(X_{\K}):= (\Lambda V_{\K},d_{\K})$
where $\Lambda V_{\K}:=\Lambda V\otimes \K$
and $d_{\K}$ is the $\K$-extension of $d$.
We say that $X$ and $Y$ have the same $\K$-homotopy type (denote $X_{\K}\simeq Y_{\K}$)
 when there is a $\K$-DGA isomorphism
$ (\Lambda V_{\K},d_{\K})\cong  (\Lambda U_{\K},d_{\K}')$
for $M(Y)=(\Lambda U,d')$.
Refer to \cite[Definition 1]{BM}.
Notice that, for $\K_1\supset \K_2$,
$ (\Lambda V_{\K_1},d_{\K_1})\cong (\Lambda U_{\K_1},d'_{\K_1})$
if $ (\Lambda V_{\K_2},d_{\K_2})\cong (\Lambda U_{\K_2},d'_{\K_2})$
(\cite{BM},\cite{MS}).\\

The $\K$-model of the above fibration $(*)$  is given as the $\K$-relative model
$$ (\K [t],0)\to (\K [t]\otimes \Lambda V_{\K},D_{\K})\overset{p_t}{\to} (\Lambda V_{\K},d_{\K})=M(X_{\K})\ \ \ (**)$$
where $|t|=2$ and $ (\K [t],0)$ is the $\K$-model of $K(\Z , 2)$.


%
\section{ C-symplecic poset structure} 


Let  ${\mathcal E}(X_{\Q})$ be 
 the group of homotopy classes  of (unpointed) homotopy 
self-equivalences of a rationalized space  $X_{\Q}$
(e.g. \cite{A}, \cite{AL}, \cite{Maru2}, \cite{Maru}, \cite{Pi}).  Let $autM$ be the group of $\K$-DGA-autmorphisms of a $\K$-DGA $M$. 
Denote $f\sim_{\K}g$ when two maps $f$ and $g$ of $autM$
are $\K$-DGA-homotopic; i.e.,
there is a $\K$-DGA-map $H:M\to M\otimes \Lambda (s,ds)_{\K}$ such that  $|s|=0$ and $|ds|=1$
with $H\mid_{s=0}=f$ and    $H\mid_{s=1}=g$.
The group of $\K$-DGA-homotopy classes of $\K$-DGA-autmorphisms of 
$M(X_{\K})= (\Lambda V_{\K},d_{\K})$
$$Aut(\Lambda V_{\K},d_{\K})=aut (\Lambda V_{\K},d_{\K})/\sim_{\K} $$
is denoted by   ${\mathcal E}(X_{\K})$ and it is equal to $aut (\Lambda V_{\K},d_{\K})/I $ 
for the {subgroup $I$  of inner \ automorphisms} of $ (\Lambda V_{\K},d_{\K})$ \cite[Proposition 6.5]{Su}.

For a fibration $\mu$ 
(see the above $(*)$) where $Y=Y_{\mu}$ is c-symplectic,  
let ${\mathcal E}(p_{\Q} )$  denote
the 
group of (unpointed) fibrewise self-equivalences 
$f_{\Q}:Y_{\Q} \to Y_{\Q}$ of the rationalized  fibration
$p_{\Q}:Y_{\Q}\to K({\Z},2)_{\Q}=K({\Q},2)$ of $(*)$. 
Thus $p_{\Q}\circ f_{\Q}=p_{\Q}$.
Futhermore,  let ${\mathcal E}(p_{\K} )$ denote  
the 
group of fibrewise self-equivalences $f_{\K}$ of 
$p_{\K}:Y_{\K}\to K({\Z},2)_{\K}=:K({\K},2)$. 
Thus $p_{\K}\circ f_{\K}=p_{\K}$ in
$$\xymatrix{
X_{\K}\ar[r]\ar[d]_{\bar{f_{\K}}}& Y_{\K}\ar[r]^{p_{\K}}\ar[d]^{f_{\K}}& K(\K ,2)\ar@{=}[d]\\
X_{\K}\ar[r]& Y_{\K}\ar[r]^{p_{\K}}& K(\K ,2),\\
}$$
where ${\bar{f_{\K}}}$ is the induced map of $f_{\K}$.
 Let ${\mathcal E}_{\mu}(X_{\K})$ denote
the image of the natural homomorphism induced by fibre restrictions
$$F_{\mu_{\K}}:{\mathcal E}(p_{\K} )\to {\mathcal E}(X_{\K})$$
with $F_{\mu_{\K}}(f_{\K})=\bar{f_{\K}}$. 
 We let
$ Aut_t(\K [t]\otimes \Lambda V_{\K},D_{\K})$ 
denote
the groups of $\K$-DGA-homotopy classes of $\K$-DGA-autmorphisms 
$f$ of $(\K [t]\otimes \Lambda V_{\K},D_{\K})$ in the above $(**)$ 
with $f(t)=t$.
Then $${\mathcal E}(p_{\K})=Aut_t(\K [t]\otimes \Lambda V_{\K},D_{\K})$$
and  ${F_{\mu}}_{\K}:{\mathcal E}(p_{\K} )\to {\mathcal E}(X_{\K})$ is equivalent to the restriction map 
$$F'_{\mu}:Aut_t(\K[t]\otimes \Lambda V_{\K},D_{\K})\to Aut (\Lambda V_{\K},d_{\K})$$
with 
$F'_{\mu}(f)(v)=p_t(f(v))$
for $v\in V$ and $p_t$ in $(**)$. 
Then $${\mathcal E}_{\mu}(X_{\K})=\operatorname{Im}F'_{\mu}.$$
Refer to \cite[Proposition 2.2]{ShY} (also see \cite{FT}, \cite{FLS}).

\begin{rem}
Recall that, for a fibration $X\to E\to B$, 
D. Gottlieb posed the problem \cite[\S 5]{Go}:
{\it Which homotopy equivalences of $X$ into itself can be extended to fibre homotopy equivalences 
of $E$ into itself ?}  
 In our case, we can propose a $\K$-realization problem on c-symplectic spaces:
{\it For which  subgroup $G$ of ${\mathcal E}(X_{\K})$
does or does not there  exist  a fibration $\mu$ such that ${\mathcal E}_{\mu}(X_{\K})=G$ ? }  
\end{rem}

\begin{lem}\label{equi}
If $\mu\cong_{\K} \mu'$ (over $K(\K ,2)$), then 
${\mathcal E}_{\mu}(X_{\K})$ and ${\mathcal E}_{\mu'}(X_{\K})$
are naturally identified 
in $ {\mathcal E}(X_{\K})$.
\end{lem}

We are interested in  the set $S:=\{{\mathcal E}_{\mu}(X_{\K}) \}_{\mu\in I}$
of subgroups  of ${\mathcal E}(X_{\K})$
for the set $I$ of isomorphism classes of $\mu_{\K}$ in $(*)$ where $Y_{\mu}$ are c-symplectic.

\begin{defn}
We define an equivalence relation 
 $\mu\sim \tau$ 
for $\mu$ and $\tau$ of $I$
if   ${\mathcal E}_{\mu}(X_{\K})= {\mathcal E}_{\tau}(X_{\K})$
in $ {\mathcal E}(X_{\K})$.
For equivalence classes $[\mu]$ and $[\tau]$,
put  $[\mu] \leq [\tau ]$
when   there is an inclusion 
$i: {\mathcal E}_{\mu}(X_{\K})\hookrightarrow {\mathcal E}_{\tau}(X_{\K})$
in ${\mathcal E}(X_{\K})$.
We define
$${\mathcal P}_{\K}(X):=(  S/\sim, \leq ) =( \{ [\mu]\} ,\leq ),$$
which is called the {\it c-symplectic poset   structure} on  $X$
over $\K$.
\end{defn}





\begin{defn}
For a field $\K$ of characteristic zero, we define the c-symplectic depth (abbr., c-s-depth) of a simply connected space $X$
over $\K$ 
by
$$\mbox{c-s-}{\rm depth}_{\K}(X)=\max \Bigl \{ n \ \large{\mid} \  [\mu_{i_1}]> [\mu_{i_2}] >\cdots >[\mu_{i_n}]\ 
\mbox{ for }\ [\mu_i]\in {\mathcal P}_{\K}(X) \Bigr \}.$$
\end{defn}

Here we note that $$\mbox{c-s-}{\rm depth}_{\K}(X)\ =
\mbox{ the height of }{\mathcal P}_{\K}(X)\ +1.$$
Suppose that $\mu_1\cong_{\K} \mu'_1$
 and $\mu_2\cong_{\K} \mu'_2$.
 If ${\mathcal E}_{\mu_1}(X_{\K})\subset {\mathcal E}_{\mu_2}(X_{\K})$,
 then we can regard  as ${\mathcal E}_{\mu'_1}(X_{\K})\subset {\mathcal E}_{\mu'_2}(X_{\K})$
from Lemma \ref{equi}. Thus we have
\begin{lem}
The non-negative integer c-s-${\rm depth}_{\K}(X)$ is a $\K$-homotopy invariant.
\end{lem}

\begin{prop}\label{finite}
When $\dim H^*(X;\Q )<\infty$, 
c-s-${\rm depth}_{\K}(X)<\infty $
for all $\K$.
\end{prop}
If $X$ is not pre-c-symplectic \cite{SY}, then ${\mathcal P}_{\K}(X)= \emptyset$, in which case we set c-s-${\rm depth}_{\K}(X):=0.$
If $X$ is pre-c-symplectic, then ${\mathcal P}_{\K}(X) \not = \emptyset$, whence clearly we have that c-s-${\rm depth}_{\K}(X) \geqq 1$.
Thus it is obvious that 
$X$ is pre-c-symplectic if and only if c-s-${\rm depth}_{\K}(X)>0$ for any $\K$.
\begin{rem} It is expected that  c-s-${\rm depth}_{\K}(X)$ measures the  abundance of c-symplectic structures
associated to a simply connected space $X$.
\end{rem}



\begin{prop}
If $X$ is pre-$c$-symplectic, then the formal dimension $fd(X)$ of $X$ is odd, where $fd(X):=\max \{n\mid  H^n(X;\Q )\neq 0\}$. 
\end{prop}
\begin{proof} Since $X$ is pre-$c$-symplectic, we have a fibration $\mu :X \to  Y_{\mu}\overset{p}{\to} K(\Z ,2)$. Then   we have a fibration one step back in  the Barratt-Puppe sequence, which is  
 a $\Omega K(\Z,2)=S^1$-fibration over $Y$ with the total space $X$, 
$S^1 \to X \to Y.$
 Let $fd(Y):=m$. Then it follows from the Gysin sequence for cohomology that $H^{m+1}(X;\Q ) \cong H^m(Y;\Q ) \not=0$ and $H^{k+1} (X;\Q )\cong H^k(Y;\Q ) = 0$ for $k>m$, thus we get $fd(X)= m +1 = fd(Y) +1$. Since $fd(Y)=m$ is even, the formal dimension $fd(X)$ of $X$ is odd.
\end{proof}
Note that $fd(X)=\max \{n\mid  H^n(X;\K )\neq 0\}$ for any $\K$
and that if $X$ is a compact manifold, then the formal dimension is the same as the (real) dimension of the manifold.

\begin{cor} If the formal dimension of $X$ is even, then c-s-${\rm depth}_{\K}(X) =0.$
\end{cor}

\begin{cor}\label{examp} 
(1) For any $n>0$, c-s-${\rm depth}_{\K}(S^{2n})=0$.\\
(2) If $X$ is a c-symplectic space, then  c-s-${\rm depth}_{\K}(X)=0$.\\
(3) If  c-s-${\rm depth}_{\K}(G/H)>0$ for a homogeneous space $G/H$, ${\rm rank }G>{\rm rank} H$.\\
(4) If c-s-${\rm depth}_{\K}(X)>0$ and c-s-${\rm depth}_{\K}(Y)>0$, then 
c-s-${\rm depth}_{\K}(X\times Y)=0$.
\end{cor}
\begin{proof}
(1) $fd(S^{2n})$ is even.\\
(2) The formal dimension of a $c$-symplectic space is even (cf. \cite[page 218, Theorem 32.6(i)]{FHT}).\\
(3) Suppose that ${\rm rank }G={\rm rank} H$.
Let $T$ be a maximal torus for both $G$ and $H$.
Since root theory tells us that 
$\dim G-\dim T$ and $\dim H-\dim T$ are both even, then $\dim G/H=\dim G-\dim H$
is also even.
Hence, c-s-${\rm depth}_{\K}(G/H)=0$.\\
(4) c-s-${\rm depth}_{\K}(X)>0$ and c-s-${\rm depth}_{\K}(Y)>0$ imply that their formal dimensions $fd(X)$ and $fd(Y)$ are both odd. By the K\"unneth theorem we have that 
$H^m(X \times Y, \Q) = \sum_{i+j=m} H^i(X, \Q) \otimes H^j(Y, \Q)$. Hence $fd(X \times Y)= fd(X) + fd(Y)$, thus $fd(X \times Y)$ is even. 
Therefore c-s-${\rm depth}_{\K}(X \times Y)=0.$
\end{proof}
\begin{thm}For any $n>0$, c-s-${\rm depth}_{\K}(S^{2n+1})=1$.
\end{thm}
\begin{proof} The result c-s-${\rm depth}_{\K}(S^{2n+1})=1$
follows from 
 the fact that 
the $\K$-Sullivan minimal model of $Y$ in $(*)$ is uniquely determined as
$(\K [t]\otimes \Lambda (v)_{\K},D_{\K})$ with $D_{\K}t=0$ and $D_{\K}v=t^{n+1}$
for $M(S^{2n+1})=(\Lambda (v),0)$.
\end{proof}
One might expect or think that $fd(X)$ is even if and only if c-s-${\rm depth}_{\K}(X) =0$, or equivalently that $fd(X)$ is odd if and only if c-s-${\rm depth}_{\K}(X) >0$. However this  turns out not to be the case, due to the
following  result:

\begin{thm}{\rm (\cite[Theorem 1.2]{SY}\label{SaYa})}  When 
$H^*(X;\Q)\cong \Lambda (v_1,v_2,\cdots ,v_n)$ with all $|v_i|$ odd
and  $1<|v_1|\leq |v_2|\leq \cdots \leq |v_n|$,
then $X$ is pre-c-symplectic 
if and only if $n$ is odd
and $|v_1|+|v_{n-1}|<|v_n|$,
 $|v_2|+|v_{n-2}|<|v_n|$, $\cdots$,
 $|v_{(n-1)/2}|+|v_{(n+1)/2}|<|v_n|$.
\end{thm}

\begin{ex}\label{exSY} Let $n$ be an odd integer $\geqq 5$ and consider the following space:
$$X:=S^{11} \times S^{15} \times \cdots \times S^{4k-1} \times \cdots \times S^{4n-1} \quad (k>3).$$
Since $n$ is odd, the formal dimension $fd(X)$ is odd, but it follows from Theorem \ref{SaYa} that c-s-${\rm depth}_{\K}(X) =0$.
\end{ex}
Theorem \ref{sp} at the end of \S 5  says  for the  symplectic group $Sp(n)$
whose  dimension $\dim Sp(n)$ is 
$n(2n+1)$ that c-s-${\rm depth}_{\K}(Sp(n)) \geqq \frac{n+1}{2}$ for any odd integer $n>3$. Thus we have the following
\begin{prop} A c-symplectic depth  can be arbitrarily large. 
\end{prop}
The above Corollary \ref{examp} (4) suggests that c-symplectic depth would not behave well with respect to taking the product of spaces. In fact, using Theorem \ref{sp} and Example \ref{exSY} we can make the following
\begin{prop} 
(1) Even if c-s-${\rm depth}_{\K}(X) =0$ and c-s-${\rm depth}_{\K}(Y)=1$, c-s-${\rm depth}_{\K}(X\times Y)$ can be arbitrarily large. 

(2) Even if  c-s-${\rm depth}_{\K}(X)= $c-s-${\rm depth}_{\K}(Y)=0$, c-s-${\rm depth}_{\K}(X\times Y)$ can be arbitrarily large.
\end{prop}
\begin{proof}
(1) Indeed, let $n$ be an odd integer $\geqq 5$ and let us consider the  rational homotopy  decomposition 
$$Sp(n)\simeq_{\Q}S^{3}\times S^{7}\times \cdots \times S^{4(n-1)-1} \times S^{4n-1} \simeq_{\Q} Sp(n-1) \times S^{4n-1}.$$
We now let $X:=Sp(n-1)$ and $Y:=S^{4n-1}$. Then c-s-${\rm depth}_{\K}(Y)= 1$. 
When $n$ is odd,  
$\dim Sp(n)=n(2n+1)$ is odd and $\dim Sp(n-1)=fd(Sp(n-1)) =fd(X)$ is even, thus c-s-${\rm depth}_{\K}(X) =0$. However it follows from Theorem \ref{sp} that c-s-${\rm depth}_{\K}(X \times Y) \geqq \frac{n+1}{2}$.\\
(2) Now consider the following:
$$Sp(n)\simeq_{\Q}\Bigl(S^{3}\times S^{7}\Bigr) \times \Bigl (S^{11} \times \cdots \times S^{4k-1} \times \cdots \times S^{4n-1} \Bigr) \, \, (k>3)$$
and set $X:=S^{3}\times S^{7}$ and $Y:= S^{11} \times \cdots \times S^{4n-1}$. Then we have that c-s-${\rm depth}_{\K}(X)= {\rm \mbox{c-s-}depth}_{\K}(Y)=0$, but we have c-s-${\rm depth}_{\K}(X \times Y)\geqq \frac{n+1}{2}.$
\end{proof}





The following theorem indicates an example that
a c-symplectic depth
strongly depends on a field $\K$.

\begin{thm}\label{CP} When $X=\C P^n\times S^{2n+3}$
with $n$ even, 
c-s-${\rm depth}_{\Q}(X)=1$ but
c-s-${\rm depth}_{\overline{\Q}}(X)=  c (n+1)$
where $c (n+1):=n_1+n_2+\cdots +n_m+1$ for  the  
prime decomposition $n+1=p_1^{n_1}p_2^{n_2}\cdots p_m^{n_m}$.
Here $\overline{\Q}$ means the  algebraic  closure  of $\Q$. 
Moreover, there is a sequence of field extensions
$$\Q \subset \K_1 \subset \K_2 \subset \cdots \subset \K_{c(n+1)}\subset \overline{\Q}$$
such that
$${\rm \mbox{c-s-}depth}_{\Q}(X)<{\rm \mbox{c-s-}depth}_{\K_1}(X)<{\rm \mbox{c-s-}depth}_{\K_2}(X)<\cdots <{\rm \mbox{c-s-}depth}_{\K_{c(n+1)}}(X). $$
\end{thm}

We see the following result from
the proof of Theorem \ref{CP}.

\begin{cor}\label{Zn}
When $X=\C P^n\times S^{2n+3}$
with $n$ even, the poset 
${\mathcal P}_{\overline{\Q}}(X)$
contains  the subgroup poset of  the group $\Z/ (n+1)\Z$. 
\end{cor}

We remark that ${\rm cat}_0(\C P^n\times S^{2n+3})={\rm cat}(\C P^n\times S^{2n+3})=n+1$.
Here ${\rm cat}$ and ${\rm cat}_0$ are respectively the Lusternik-Schnirelmann (LS) category 
and the rational LS category ${\rm cat}_0(X):={\rm cat}(X_{\Q})$ of a space $X$ (e.g., see \cite{FHT}).
Notice that ${\rm cat}_0(X)={\rm cat}(X_{\K})$ 
for all field $\K$ \cite{He}.
On the other hand, for a cyclotomic field $\K$, the proof of Theorem \ref{CP} indicates that
${\mathcal P}_{\K}(X)$ often presents more informations on a classification 
of $\{ Y_{\mu}\}$ in  $(*)$ in \S 1 than ${\mathcal P}_{\Q}(X)$.
In connection with LS categories, we would like to pose the following
\begin{qu}
 For any field $\K$, is
c-s-${\rm depth}_{\K}(X)\leq {\rm cat}_0(X)\ ?$
\end{qu}


\begin{rem}
Let $Y$ be a simply connected c-symplectic space.
We 
define 
$$l_X(Y)_{\K}:=
\begin{cases}
\max \{ n\mid [\mu_1]>[\mu_2]>\cdots >[\mu_{n-1}]>[\mu_{n}]=[Y] \ in \ {\mathcal P}_{\K}(X)\},\\
0, \quad \text{if there exists no  fibration $X\to Y\to K(\Z ,2)$}.
\end{cases}
$$
If $\dim \pi_2(Y)\otimes \Q=1$,
such a space $X$ uniquely exists. 
We speculate that $l_X(Y)_{\K}$ must 
reflect a certain complexity of c-symplectic structure of $Y$.
This definition is something  like the co-height  of a prime ideal (e.g.\cite{Z}) by adding $1$.
Here the co-height  of a prime ideal $p$ in a ring $R$  is defined
as the largest $n$ for which there exists
a chain of different prime ideals $p\subset p_1\subset \cdots \subset p_n\neq R$.

\end{rem}





\section{Proofs}

\noindent{\it Proof of Lemma \ref{equi}.}
Let $\mu' :X \to  Y_{\mu'}\overset{p'}{\to} K(\Z ,2)$
be a fibration with $\mu\cong_{\K} \mu'$.
Suppose that
$$\psi :M({Y_{\mu}}_{\K})=({\K} [t]\otimes \Lambda V,D)\to ({\K} [t]\otimes \Lambda V,D')=M({Y_{\mu'}}_{\K})$$
is an isomorphism over $({\K} [t],0)$.
Then the restriction map $\overline{\psi} :(\Lambda V_{\K},\overline{D}_{\K})\to 
(\Lambda V_{\K},\overline{D'}_{\K})$ is an isomorphism.
We define 
 $$ad_{\psi}: Aut_t(\K [t]\otimes \Lambda V_{\K},D_{\K})\to Aut_t(\K [t]\otimes \Lambda V_{\K},D_{\K}')$$
by $ad_{\psi}(f)=\psi\circ f\circ \psi^{-1}$, which is well-defined.
Then we get the following  commutative diagram of groups:
$$\xymatrix{
{\mathcal E}(X_{\K})\ar[r]^{ad_{\overline{\psi}}} & {\mathcal E}(X_{\K})\\
{\mathcal E}(p_{\K})\ar[u]^{res.}\ar[r]^{ad_{\psi}}& {\mathcal E}(p_{\K}'),\ar[u]_{res.}
}$$
where the vertical maps are the restriction maps.
Thus ${\mathcal E}_{\mu}(X_{\K})$
is identified 
with  ${\mathcal E}_{\mu'}(X_{\K})$
by the isomorphism $ad_{\overline{\psi}}$
given by $ad_{\overline{\psi}}(f)=\overline{\psi}\circ f\circ \overline{\psi}^{-1}$.
\hfill\qed\\

\noindent{\it Proof of Proposition \ref{finite}}.
Suppose that there is a sequence
$${\mathcal E}(X_{\K})\supsetneq  {\mathcal E}_{\mu_1}(X_{\K})\supsetneq  {\mathcal E}_{\mu_2}(X_{\K})\supsetneq \cdots \supsetneq  {\mathcal E}_{\mu_m}(X_{\K}) $$
for some $m$.
It is equivalent to  a sequence 
$$Aut(\Lambda V_{\K},d_{\K})\supsetneq  {\rm Im} F'_{\mu_1}\supsetneq  {\rm Im}F'_{\mu_2}\supsetneq \cdots \supsetneq  {\rm Im}F'_{\mu_m}$$
for the restriction maps $F'_{\mu_i}: Aut_t(\K[t]\otimes \Lambda V_{\K},{D_i}_{\K})\to Aut (\Lambda V_{\K},d_{\K})$
in \S 3.
Then we have a sequence of $\K$-algebraic groups  $$aut(\Lambda V_{\K},d_{\K})\supsetneq  G_{\mu_1}\supsetneq  G_{\mu_2}\supsetneq \cdots \supsetneq 
 G_{\mu_m}$$
 where $ G_{\mu_i}$ are the images of the restrictions sending $t$ to zero
 $$aut_t(\K[t]\otimes \Lambda V_{\K},{D_i}_{\K})\to aut (\Lambda V_{\K},d_{\K})$$
with $G_{\mu_i}/_{\sim_{\K}} = {\rm Im}F'_{\mu_i}$.
 Since $\dim H^*(\Lambda V_{\K},d_{\K})=\dim H^*(X;\K)<\infty$,
being $\K$-DGA-autmorphisms induces  $$aut(\Lambda V_{\K},d_{\K})=aut(\Lambda V^{\leq n}_{\K},d_{\K}\mid_{V^{\leq n}_{\K}})$$
for a sufficiently  large $n$.
Here $V^{\leq n}_{\K}$ means the subspace $\{ v\in V_{\K};\  |v|\leq n\}$ of $V_{\K}$,
which is finite-dimensional over $\K$.
The latter is an algebraic  matrix  group (defined by polynomial equations in the entries)
 in  the general $\K$-linear group 
$GL(N,\K )$ for a sufficiently large $N$ \cite[page 294]{Su}.
From the Noetherian property of descending chain condition, 
the integer $m$ is  bounded by $N$.  
Thus  we have 
c-s-${\rm depth}_{\K}(X)<\infty$.
\hfill\qed\\






\noindent{\it Proof of Theorem \ref{CP}.}
Let $M(X_{\K})=(\Lambda (x,y,z)_{\K},d_{\K})$
with
$$|x|=2, \ |y|=2n+1,\  |z|=2n+3 \quad \text {and}  \quad
d_{\K}(x)=d_{\K}(z)=0, \ d_{\K}(y)=x^{n+1}.$$
Then we have
$$  {\mathcal E} ({X}_{\K})=Aut(\Lambda (x,y,z)_{\K},d_{\K})=\Bigg\{ \begin{pmatrix}
a & &\\
 & a^{n+1}&\\
 &&b\\
 \end{pmatrix}\ \ | \ \ a,b\in \K-0\Bigg\},
$$
where
$f(x)=ax$, $f(y)=a^{n+1}y$ and $f(z)=bz$ for $f\in Aut(\Lambda (x,y,z)_{\K},d_{\K})$.
Let us denote
the $\K$-relative models of $(*)$  
by $ (\K [t]\otimes \Lambda (x,y,z)_{\K},D_{\mu})=M({Y_{\mu}}_{\K})$ with 
$$D_{\mu}(t)=D_{\mu}(x)=0,\ \
 D_{\mu}(y)=x^{n+1}+x^{i}t^{n-i+1}+t^{n+1} \ \mbox{  \ and}\ \ D_{\mu}(z)=xt^{n+1}$$
where $i$ is one of the divisors of $n+1$ but not $n+1$ itself
or $i=0$.
Then $fd(Y_{\mu})=4n+2$ and $[t^{2n+1}]\neq 0$
in $H^*(Y_{\mu};\K )=\K [t,x]/(x^{n+1}+x^{i}t^{n-i+1}+t^{n+1} ,xt^{n+1})$.
Let  
$$N_{\mu}:=
\begin{cases}
i  & \mbox{ when }i\neq 0 \\
n+1 & \mbox{ when }i=0.
\end{cases}
$$

\noindent
Then we have   
\begin{eqnarray*}{\mathcal E}_{\mu}(X_{\K})&=&{\rm Im}(F'_{\mu}: 
Aut_t(\K[t]\otimes \Lambda (x,y,z)_{\K},D_{\mu})\to Aut (\Lambda (x,y,z)_{\K},d_{\K}))\\
&=&\Bigg\{ \begin{pmatrix}
a & &\\
 & 1&\\
 &&a\\
 \end{pmatrix}\ \ | \ \  a^{N_{\mu}}=1\mbox{  for } a\in \K-0\Bigg\}\ \ 
(f(x)=ax,\ f(y)=y,\ f(z)=az)\\
&\cong &
  \{ a\in \K -0\ |\  \ a^{N_{\mu}}=1 \}
  \end{eqnarray*}
 as groups.
Suppose that
 $\K\supset \Q (e^{2\pi i/(n+1)}) $.
Then ${\mathcal E}_{\mu}(X_{\K})$ is isomorphic to $\Z_{N_{\mu}}=\Z /N_{\mu}\Z$.
Notice that  ${\mathcal E}_{\mu_{i}}(X_{\K})\subset  {\mathcal E}_{\mu_j}(X_{\K})$
if and only if $N_j|N_i$, i.e., $N_j$ is a divisor of $N_i$.
Therefore there is a sequence
$$(\{ a\in \K\ |\  \ a^{N_{\mu}}=1 \}=) \ {\mathcal E}_{\mu_{1}}(X_{\K})\supsetneq {\mathcal E}_{\mu_{2}}(X_{\K})\supsetneq \cdots
\supsetneq {\mathcal E}_{\mu_{c(n+1)}}(X_{\K})\ (=\{ id_{X_{\K}}\}).$$
Since any subgroup $ {\mathcal E}_{\mu}(X_{\K})$
of $ {\mathcal E}(X_{\K})$ is isomorphic to some subgroup of $\Z_{n+1}$
from the degree argument of $t,\ x,\ y$ and $z$,
 it has the  maximal length of  inclusions of subgroups.
Thus we have that $$ \mbox{c-s-}{\rm depth}_{{\K}}(X)\ (= \mbox{c-s-}{\rm depth}_{\overline{\Q}}(X))\ =c (n+1).$$


Moreover the sequence of length $c(n+1)$
$$ \Z_{p_1^{n_1}p_2^{n_2}\cdots p_m^{n_m}} \supsetneq     \Z_{p_1^{n_1-1}p_2^{n_2}\cdots p_m^{n_m}} \supsetneq 
 \Z_{p_1^{n_1-2}p_2^{n_2}\cdots p_m^{n_m}} \supsetneq  
\cdots  \supsetneq  \Z_{p_m}  \supsetneq  \{ 0\}$$
identified with 
the proper sequence of subgroups for the fibration $\mu$ of the differential $D_{\mu}$ with 
$D_{\mu}(t)=D_{\mu}(x)=0$,
 $D_{\mu}(y)=x^{n+1}+t^{n+1}$  and $D_{\mu}(z)=xt^{n+1}$:
$$  {\mathcal E}_{\mu }(X_{\Q (p_1^{n_1}p_2^{n_2}\cdots p_m^{n_m})})\supsetneq  {\mathcal E}_{\mu }(X_{\Q ({p_1^{n_1-1}p_2^{n_2}\cdots p_m^{n_m}})})\supsetneq \cdots \supsetneq  
{\mathcal E}_{\mu }(X_{\Q ({p_m})})\supsetneq  
{\mathcal E}_{\mu }(X_{\Q }),$$
where $\Q (q)$ means the extension field $\Q (e^{2\pi i/q}) $ of $\Q$ by adding 
a primitive $q$th root  of unity.
Thus there is the sequence $${\rm \mbox{c-s-}depth}_{\Q}(X)<{\rm \mbox{c-s-}depth}_{\K_1}(X)<{\rm \mbox{c-s-}depth}_{\K_2}(X)<\cdots <{\rm \mbox{c-s-}depth}_{\K_{c(n+1)}}(X)$$
for $\K_1=\Q ({p_1})$, $\K_2=\Q ({p_1p_2})$ or $\Q ({p_1^2})$,  $\cdots$,  $\K_{c(n+1)}=\Q (p_1^{n_1}p_2^{n_2}\cdots p_m^{n_m})$.
 \hfill\qed\\




\section{Examples }
In this section,  
for odd integers $n_1\leq n_2\leq \cdots \leq n_k$, let
$$M(S^{n_1}\times S^{n_2}\times \cdots \times S^{n_k})=(\Lambda V,0)=(\Lambda (v_1,v_2,..,v_k),0)$$
with $|v_i|=n_i$ for all $i$.
In the following Examples 5.1, 5.2, and 5.3(1),
the poset structure of ${\mathcal P}_{\K}(X)$
does not depend on $\K$.

\begin{ex}(c-s-${\rm depth}_{\K}(X)=1$).
When $X$ is 
$$(a)\, \, S^{3}\times S^{3}\times S^{7}, \qquad \qquad \qquad \qquad \qquad \qquad $$
$$(b)\, \, S^{7}\times S^{9}\times S^{11}\times S^{13}\times S^{23}, \qquad \qquad \qquad$$
$$(c)\, \, S^{9}\times S^{9}\times S^{11}\times S^{13}\times S^{15}\times S^{17}\times S^{29}, \ $$
$$(d)\, \, S^{9}\times S^{11}\times S^{13}\times S^{15}\times S^{17}\times S^{19}\times S^{31}, $$
  then the Hasse diagrams of ${\mathcal P}_{\K}(X)$
 are respectively one point, two points, three points  and four points:
$$(a)\ \bullet,\ \ \ \ (b)\ \bullet \bullet,\ \ \ \ \ (c)\ \bullet \bullet \ \bullet,\ \ \ \ \ 
(d)\ \bullet\bullet\bullet\bullet
$$
For example,  the differential $D$ in the case of $(d)$, $\{ \bullet\bullet\bullet \ \bullet\}=\{ \mu_1,..,\mu_4\}$,
 is given by \\

$Dv_1=\cdots =Dv_6=0  \qquad \text{and} \qquad \qquad $

$\mu_1$: $ D v_7=v_1v_6t^2+v_2v_5t+v_3v_4t^2+t^{16 },\ \ \ \ \ $

$\mu_2$: $ D v_7=v_1v_6t^2+v_2v_4t^3+v_3v_5t+t^{16 },\ \ \ \ $

$\mu_3$: $ D v_7=v_1v_5t^3+v_2v_6t+v_3v_4t^2+t^{16 },\  \ \ \ $

$\mu_4$: $ D v_7=v_1v_4t^4+v_2v_6t+v_3v_5t+t^{ 16}.\ \ \ \ \ \ $\\

Notice that $Dv_1=\cdots =Dv_{k-1}=0$ when c-s-${\rm depth}_{\K}(S^{n_1}\times S^{n_2}\times \cdots \times S^{n_k})=1$
in general.
Let  \begin{eqnarray*}
  {\mathcal E} ({X}_{\K})=Aut(\Lambda V_{\K},0)&=&\Bigg\{ \begin{pmatrix}
a & &&&&&\\
 & b&&&&&\\
 &&c&&&&\\
 &&&d&&&\\
 &&&&e&&\\
 &&&&&f&\\
 &&&&&&g\\
\end{pmatrix}\ \ \mid  \ \ a,b,c,d,e,f,g\in \K-0\Bigg\}\\
&=&\Big\{ {\rm diag}(a,b,c,d,e,f,g) \mid  a,b,c,d,e,f,g\in \K-0\Big\},  
 \end{eqnarray*}
where  $h(v_1)=av_1$, $h(v_2)=bv_2$, $h(v_3)=cv_3$, $h(v_4)=dv_4$,
 $h(v_5)=ev_5$,  $h(v_6)=fv_6$
and  $h(v_7)=gv_7$
for $h\in  {\mathcal E} ({X}_{\K})$.
Then 
$$  {\mathcal E}_{\mu_1} ({X}_{\K})=\{ {\rm diag}(a,b,c,c^{-1},b^{-1},a^{-1},1) \}  $$
$$  {\mathcal E}_{\mu_2} ({X}_{\K})=\{ {\rm diag}(a,b,c,b^{-1},c^{-1},a^{-1},1) \}  $$
$$  {\mathcal E}_{\mu_3} ({X}_{\K})=\{ {\rm diag}(a,b,c,c^{-1},a^{-1},b^{-1},1) \}  $$
$$  {\mathcal E}_{\mu_4} ({X}_{\K})=\{ {\rm diag}(a,b,c,a^{-1},c^{-1},b^{-1},1) \}  $$
for $a,b,c\in \K$
as subgroups of $  {\mathcal E} ({X}_{\K})$.
 \end{ex}



\begin{ex}\label{ex3}(c-s-${\rm depth}_{\K}(X)=2$).
(1) Let  $X(n)=S^{3}\times S^{5}\times S^{7}\times S^{9}\times S^{n}$.

When (a) $n=13$, (b) $n=15$ and  (c) $n=17$, the Hasse diagrams of ${\mathcal P}_{\K}(X(n))$ are respectively given as 
$$(a)\ \xymatrix{\mu_1 \ar@{-}[d]\\
\mu_2
}\ \ \ \ \ \ \ \ \ \ \ \ 
(b) \ \xymatrix{\mu_1 \ar@{-}[d]&\mu_3 \ar@{-}[d]\\
\mu_2 &\mu_4 
}\ \ \ \ \ \ \ \ \ 
(c) \ \xymatrix{\mu_1\ar@{-}[d]&\mu_3\ar@{-}[d]& \bullet \ \mu_5\\
\mu_2&\mu_4&
}
$$

Here the point $\mu_1$ of $(a)$ is given by 
$Dv_1=Dv_2=Dv_3=Dv_4=0$
and $Dv_5=v_1v_4t+v_2v_3t+t^{7}$.
On the other hand, $\mu_2$  is given by 
$Dv_4=v_1v_2t$ and $Dv_i$ ($i\neq 4$) are the same as  $\mu_1$.

Next, the points $\mu_k$ with 
$k=1, 2, 3, 4, 5$ of $(b)$ and   $(c)$  are given by the following differentials
$$Dv_1=Dv_2=Dv_3=0 \qquad \text {and}$$
\begin{center}{\begin{tabular}{|c|c|c|}
\hline
$(b)$ & $Dv_4$ &$Dv_5$ \\
\hline
$\mu_1$   & $0$         &$v_1v_4t^{2}+v_2v_3t^{3}+t^{8}$  \\
\hline
$\mu_2$   & $v_1v_2t$         & $v_1v_4t^{2}+v_2v_3t^{3}+t^{8}$\\
\hline
$\mu_3$   & $0$         & $v_1v_3t^3+v_2v_4t^2+t^8$\\
\hline
$\mu_4$   & $v_1v_2t$         & $v_1v_3t^3+v_2v_4t^2+t^8$\\
\hline
\end{tabular}
}
{\begin{tabular}{|c|c|c|}
\hline
$(c)$ & $Dv_4$ &$Dv_5$ \\
\hline
$\mu_1$   & $0$         &$v_1v_4t^{3}+v_2v_3t^{4}+t^{9}$  \\
\hline
$\mu_2$   & $v_1v_2t$         & $v_1v_4t^{3}+v_2v_3t^{4}+t^{9}$\\
\hline
$\mu_3$   & $0$         & $v_1v_3t^4+v_2v_4t^3+t^9$\\
\hline
$\mu_4$   & $v_1v_2t$         & $v_1v_3t^4+v_2v_4t^3+t^9$\\
\hline
$\mu_5$   & $0$         & $v_1v_2t^5+v_3v_4t+t^9$\\
\hline
\end{tabular}
}
\end{center}
Let
$$  {\mathcal E} ({X(n)}_{\K})=\{ ({\rm diag}(a,b,c,d,e),\lambda ) \mid  a,b,c,d,e\in \K-0, \lambda\in \K\}, $$
where  $f(v_1)=av_1$, $f(v_2)=bv_2$, $f(v_3)=cv_3$, $f(v_4)=dv_4$ and 
 $f(v_5)=ev_5+\lambda v_1v_2v_3$ or $f(v_5)=ev_5+\lambda v_1v_2v_4$
for $f\in  {\mathcal E} ({X(n)}_{\K})$.

For $(b)$,
\begin{eqnarray*}
  {\mathcal E}_{\mu_1} ({X({15})}_{\K})&=&\{ ({\rm diag}(a,b,b^{-1},a^{-1},1),\lambda ) \}  \\
  {\mathcal E}_{\mu_2} ({X({15})}_{\K})&=&\{ ({\rm diag}(a,a^{-2},a^{2},a^{-1},1), \lambda )\} \\ 
   {\mathcal E}_{\mu_3} ({X({15})}_{\K})&=&\{ ( {\rm diag}(a,b,a^{-1},b^{-1},1) \}, \lambda ) \\
  {\mathcal E}_{\mu_4} ({X({15})}_{\K})&=&\{ ({\rm diag}(b^{-2},b,b^{2},b^{-1},1),\lambda ) \}
\end{eqnarray*}

For $(c)$,
\begin{eqnarray*}
  {\mathcal E}_{\mu_1} ({X({17})}_{\K})&=&\{ ({\rm diag}(a,b,b^{-1},a^{-1},1), \lambda )\}  \\
  {\mathcal E}_{\mu_2} ({X({17})}_{\K})&=&\{ ({\rm diag}(a,a^{-2},a^{2},a^{-1},1),\lambda ) \} \\
  {\mathcal E}_{\mu_3} ({X({17})}_{\K})&=&\{ ({\rm diag}(a,b,a^{-1},b^{-1},1), \lambda )\}  \\
  {\mathcal E}_{\mu_4} ({X({17})}_{\K})&=&\{ ({\rm diag}(b^{-2},b,b^{2},b^{-1},1), \lambda )\} \\
  {\mathcal E}_{\mu_5} ({X({17})}_{\K})&=&\{ ({\rm diag}(a,a^{-1},c,c^{-1},1),\lambda ) \}  
\end{eqnarray*}  
for $a,b,c\in \K -0$ and $\lambda\in \K$.
\vspace{0.5cm}

(2) Let us consider (a) $X=S^{3}\times S^{5}\times S^{7}\times S^{11}\times S^{15}$ and 
 (b) $X=S^{7}\times S^{9}\times S^{11}\times S^{13}\times S^{41}$. Then
 the Hasse diagrams of ${\mathcal P}_{\K}(X)$ are respectively as follows:
$$
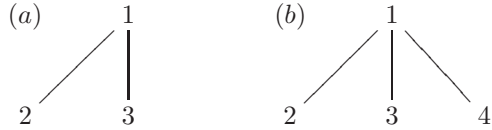
\xymatrix{
(a)&1\ar@{-}[d]\ar@{-}[ld]\\
2&3
}\  \ \ \ \ \ \ \ \ \ \ \ \ \ 
\xymatrix{
(b)&1\ar@{-}[d]\ar@{-}[ld]\ar@{-}[rd]&\\
2&3&4
}\ 
$$ Here the differentials are  given by 

$$Dv_1=Dv_2=Dv_3=0 \quad \text {and}$$
\begin{center}
{\begin{tabular}{ |c|c|c|}
\hline
$(a)$ & $Dv_4$ &$Dv_5$ \\
\hline
$1$    & $0$         &$v_1v_4t+v_2v_3t^2+t^8$  \\
\hline
$2$     & $v_1v_2t^2$         & $v_1v_4t+v_2v_3t^2+t^8$\\
\hline
$3$     & $v_1v_3t$         & $v_1v_4t+v_2v_3t^2+t^8$\\
\hline
\end{tabular}
}\ \ \ \ \ \  \ \ \ {\begin{tabular}{ |c|c|c|}
\hline
$(b)$ & $Dv_4$ &$Dv_5$ \\
\hline
$1$    & $0$         &$v_1v_2v_3v_4t+t^{21}$  \\
\hline
$2$     & $0$         & $v_1v_2t^{13}+v_3v_4t^9+t^{21}$\\
\hline
$3$     & $0$         & $v_1v_3t^{12}+v_2v_4t^{10}+t^{21}$\\
\hline
$4$     & $0$         & $v_1v_4t^{11}+v_2v_3t^{11}+t^{21}$\\
\hline
\end{tabular}
}
\end{center}
(Note: From here on we simply denote $k$ for $\mu_k$
in the Hasse diagram  and the left column of the table.)

In the case of $(a)$, 
by degree arguments we have
$$  {\mathcal E} (X_{\K})=\{ ({\rm diag}(a,b,c,d,e),\lambda ) \mid  a,b,c,d,e\in \K-0, \lambda \in \K \}, $$
where  $f(v_1)=av_1$, $f(v_2)=bv_2$, $f(v_3)=cv_3$, $f(v_4)=dv_4$
and $f(v_5)=ev_5+\lambda v_1v_2v_3$ for $f\in  {\mathcal E} ({X}_{\K})$.
 As subgroups of $  {\mathcal E} ({X}_{\K})=\{ {\rm diag}(a,b,c,d,e)\}$, we have
\begin{eqnarray*}
  {\mathcal E}_{\mu_1} (X_{\K})&=& \{ ({\rm diag}(a,b,b^{-1},a^{-1},1),\lambda )\}\\
{\mathcal E}_{\mu_2} (X_{\K})&=& \{ ({\rm diag}(a,a^{-2},a^{2},a^{-1},1),\lambda )\}\\
  {\mathcal E}_{\mu_3} (X_{\K})&=& \{ ({\rm diag}(a,a^2,a^{-2},a^{-1},1),\lambda )\}
\end{eqnarray*}
 In the case of $(b)$, similarly as subgroups of $  {\mathcal E} ({X}_{\K})$, we have
\begin{eqnarray*} 
  {\mathcal E}_{\mu_1} (X_{\K})&=& \{ ({\rm diag}(a,b,c,d,1)\mid abcd=1 \}\\
{\mathcal E}_{\mu_2} (X_{\K})&=& \{ ({\rm diag}(a,a^{-1},c,c^{-1},1) \}\\
  {\mathcal E}_{\mu_3} (X_{\K})&=& \{ ({\rm diag}(a,b,a^{-1},b^{-1},1)\}\\
  {\mathcal E}_{\mu_4} (X_{\K})&=& \{ ({\rm diag}(a,b,b^{-1},a^{-1},1)\}
 \end{eqnarray*}
 for $a,b,c,d\in \K -0$.
\end{ex}

\begin{ex}\label{ex4}(c-s-${\rm depth}_{\K}(X)=3$).\\
 (1)  Let us consider (a)  $X=S^{3}\times S^{5}\times S^{9}\times S^{13}\times S^{17}$
and 
 (b) $X=S^{7}\times S^{9}\times S^{11}\times S^{17}\times S^{45}$. 
Then we will show that the Hasse diagrams of ${\mathcal P}_{\K}(X)$ are  
respectively as follows:
$$
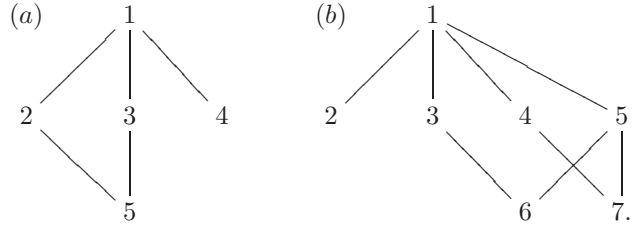
\xymatrix{(a)&1\ar@{-}[ld]\ar@{-}[rd]\ar@{-}[d]&\\
2\ar@{-}[rd] &3\ar@{-}[d] &4\\
&5&
}\ \ \ \ \ \  \ \ \xymatrix{
(b)&1\ar@{-}[d]\ar@{-}[ld]\ar@{-}[rd]\ar@{-}[rrd]&&\\
2&3\ar@{-}[rd]&4\ar@{-}[rd]&5\ar@{-}[d]\ar@{-}[ld]\\
&&6&7.
}\
$$
The differentials in $(a)$ and $(b)$   are  given by 
 $$Dv_1=Dv_2=0 \, \, (Dv_3=0 \, \, \text{for} \, \,  (b)) \, \, \text{and}$$   
\begin{center}
{\begin{tabular}{ |c|c|c|c|}
\hline
$(a)$  &$Dv_3$& $Dv_4$ &$Dv_5$ \\
\hline
$1$        & $0$ & $0$         &$v_1v_4t+v_2v_3t^2+t^9$  \\
\hline
$2$        & $v_1v_2t$ & $0$         & $v_1v_4t+v_2v_3t^2+t^9$\\
\hline
$3$        & $0$ & $v_1v_3t$         & $v_1v_4t+v_2v_3t^2+t^9$\\
\hline
$4$        &$0$ & $v_1v_2t^3$         & $v_1v_4t+v_2v_3t^2+t^9$\\
\hline
$5$        &$v_1v_2t$ & $v_1v_3t$         & $v_1v_4t+v_2v_3t^2+t^9$\\
\hline
\end{tabular}
}\  {\begin{tabular}{ |c|c|c|}
\hline
$(b)$ & $Dv_4$ &$Dv_5$ \\
\hline
$1$    & $0$         &$v_1v_2v_3v_4t+t^{23}$  \\
\hline
$2$     & $0$         & $v_1v_2t^{15}+v_3v_4t^9+t^{23}$\\
\hline
$3$     & $0$         & $v_1v_3t^{14}+v_2v_4t^{10}+t^{23}$\\
\hline
$4$     & $0$         & $v_1v_4t^{11}+v_2v_3t^{13}+t^{23}$\\
\hline
$5$    & $v_1v_2t$         &$v_1v_2v_3v_4t+t^{23}$  \\
\hline
$6$     & $v_1v_2t$         & $v_1v_3t^{14}+v_2v_4t^{10}+t^{23}$\\
\hline
$7$     & $v_1v_2t$         & $v_1v_4t^{11}+v_2v_3t^{13}+t^{23}$\\
\hline
\end{tabular}
}
\end{center}
In the case of 
$ (a)$, we have $$ {\mathcal E} (X_{\K})=\{ ({\rm diag}(a,b,c,d,e),\lambda ) \mid  a,b,c,d,e\in \K-0, \lambda \in \K \}$$
where  $f(v_1)=av_1$, $f(v_2)=bv_2$, $f(v_3)=cv_3$, $f(v_4)=dv_4$,
 $f(v_5)=ev_5+\lambda v_1v_2v_3$ for $f\in  {\mathcal E} ({X}_{\K})$.
In the case of $(b)$, we have 
$$ {\mathcal E} (X_{\K})=\{ {\rm diag}(a,b,c,d,e) \mid  a,b,c,d,e\in \K-0 \}. $$
The subgroups of  $ {\mathcal E} (X_{\K})$ for $(a)$ and $(b)$ are as follows:
\begin{center}
{\begin{tabular}{ |c|c|}
\hline
$(a)$  &$ {\mathcal E}_{\mu_i} (X_{\K}) $ \\
\hline
$1$        &$(a,b,b^{-1},a^{-1},1,\lambda ) $  \\
\hline
$2$        & $(b^{-2},b,b^{-1},b^{2},1,\lambda ) $\\
\hline
$3$        &$(a,a^{2},a^{-2},a^{-1},1,\lambda ) $ \\
\hline
$4$        &$(a,a^{-2},a^{2},a^{-1},1,\lambda ) $\\
\hline
$5$        &$(a,a^{2},a^{3},a^{4},1,\lambda ), \ a^5=1$\\
\hline
\end{tabular}
}\  {\begin{tabular}{ |c|c|}
\hline
$(b)$ & $ {\mathcal E}_{\mu_i} (X_{\K}) $ \\
\hline
$1$    & $(a,b,c,d,1), \ abcd=1 $  \\
\hline
$2$     & $(a,a^{-1},c,c^{-1},1) $\\
\hline
$3$     & $(a,b,a^{-1},b^{-1},1) $\\
\hline
$4$     & $(a,b,b^{-1},a^{-1},1) $\\
\hline
$5$    &   $(a,b,c,ab,1), \ a^2b^2c=1  $\\
\hline
$6$     & $(b^{-2},b,b^{2},b^{-1},1) $\\
\hline
$7$     &$(a,a^{-2},a^{2},a^{-1},1) $ \\
\hline
\end{tabular}
}
\end{center}
for $a,b,c\in \K -0$ and $\lambda \in \K$.
In the case $(a)$,  $ {\mathcal E}_{\mu_5} (X_{\K})\cong \Z_5\times \K$ when $\K\supset \Q (e^{2\pi i/5})$
and it is isomorphic to $\K$ when $\K\not\supset \Q (e^{2\pi i/5})$.

(2) Let $X=\C P^{14}\times S^{31}$. 
It is the case of $n=14$ in Theorem \ref{CP}.
Let $$M(X_{\K})=(\Lambda (x,y,z)_{\K},d_{\K}) \quad {
with}\quad |x|=2, |y|=29, |z|=31, d_{\K}y=x^{15}.$$
Since $n+1=15(=3\cdot 5)$, 
assume  $\K\supset \Q (e^{2\pi i/15}) $.
Let $(\K[t]\otimes \Lambda (x,y,z)_{\K},{D_i})$ ($i=1,2,3,4$) be the
relative model with $$D_i(x)=0,\ \ D_i( z)=xt^{15}$$ and
the differential of $y$ being  one of the following:
 \begin{eqnarray*}
D_1(y)&=&x^{15}+t^{15},\\
D_2(y)&=&x^{15}+x^3t^{12}+t^{15},\\
D_3(y)&=&x^{15}+x^5t^{10}+t^{15},\\
D_4(y)&=&x^{15}+xt^{14}+t^{15}.
\end{eqnarray*}
Then,  
we 
have the subgroups of  $ {\mathcal E} (X_{\K})=\{ {\rm diag}(a,a^{15},b)\mid a,b\in \K-0\}\cong ({\K}-0)^{\times 2}$,
which are  the following:
 \begin{eqnarray*}
  {\mathcal E}_{\mu_1} (X_{\K})&=&\{ {\rm diag}(a,1,a)\mid a^{15}=1\}=\{ a\in {\K}-0\mid a^{15}=1\}\cong \Z /15\Z ,\\
{\mathcal E}_{\mu_2} (X_{\K})&=&\{ {\rm diag}(a,1,a)\mid a^{3}=1\}=\{ a\in {\K}-0\mid a^{3}=1\}\cong \Z /3\Z ,\\
  {\mathcal E}_{\mu_3} (X_{\K})&=&\{ {\rm diag}(a,1,a)\mid a^{5}=1\}=\{ a\in {\K}-0\mid a^{5}=1\}\cong \Z /5\Z ,\\
 {\mathcal E}_{\mu_4} (X_{\K})&=&\{ {\rm diag}(1,1,1) \}\cong \{ 0\}.
  \end{eqnarray*}
Hence the Hasse diagram contains  the following:
$$
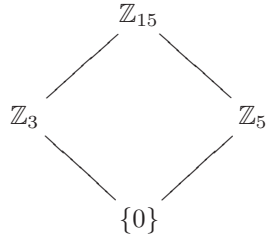
\xymatrix{&\Z_{15} \ar@{-}[ld]\ar@{-}[rd]&\\
\Z_{3}\ar@{-}[rd] & &\Z_5\ar@{-}[ld]\\
&\{ 0\}&}
$$

(3) Let $X=\C P^{8}\times S^{19}$. 
It is the case of $n=8$ in Theorem \ref{CP}. Then, when $\K\supset \Q (e^{2\pi i/9})$, the Hasse diagram 
contains the following:
$$
\xymatrix{&\Z_{9} \ar@{-}[d]\\
&\Z_{3}\ar@{-}[d]\\
&\{ 0\}
}$$


\end{ex}

\begin{ex}\label{ex2}(c-s-${\rm depth}_{\K}(X)=4$).  
 When $X=S^3\times S^5\times S^9\times S^{15}\times S^{33}$ (cf.\cite[Example 2.8]{SY}),
we have the following two cases: \\
(1)  Case of  $\K\supset \Q (e^{2\pi i/5}) $. Then $\sharp {\mathcal P}_{\K}(X)=20$. \\
(2) The other  fields $\K$. Then  $\sharp {\mathcal P}_{\K}(X)=19$.

 Indeed, the 20 (19)   types\footnote{
For example, the relative model   ${\mathcal M_{\tau}}$ with  $ Dv_3=Dv_4=0,\ Dv_5=v_1v_2v_3v_4t+v_2v_3t^{10}+t^{17}$
has same  self-equivalence  as one of  $i=1$ in the above table; i.e., $ {\mathcal E}_{\tau}(X_{\K})={\mathcal E}_{\mu_{1}}(X_{\K})$.
So ${\mathcal M}_{\tau}$  is not noted.
} of c-symplectic models $ \{  [\mu_1], .., [\mu_{20}]\}$ are given as 
$$M(Y_n)=(\K [t]\otimes \Lambda (v_1,v_2,v_3,v_4,v_5)_{\K},D_n)$$
 with  $|v_1|=3,\ |v_2|=5,\ |v_3|=9,\ |v_4|=15,\ |v_5|=33$ and
the differentials are given by  $D_nv_1=D_nv_2=0$ and \\
\begin{center}
{\begin{tabular}{|c|c |c|c|c|}
\hline
$n$ &$D_nv_3$& $D_nv_4$ &$D_nv_5$ & $\dim H^*(Y_n;\K )$\\
\hline
$1$&$0$& $0$ &$v_1v_4t^8+v_2v_3t^{10}+t^{17}$ &$272$ \\
\hline
$2$&$0$& $v_1v_2t^4$ &$v_1v_4t^8+v_2v_3t^{10}+t^{17}$ & $220$\\
\hline
$3$&$0$& $v_1v_3t^2$ &$v_1v_4t^8+v_2v_3t^{10}+t^{17}$  &$212$\\
\hline
$4$&$v_1v_2t$& $0$ &$v_1v_4t^8+v_2v_3t^{10}+t^{17}$  &$209$\\
\hline
$5$&$v_1v_2t$& $v_1v_3t^2$ &$v_1v_4t^8+v_2v_3t^{10}+t^{17}$ & $149$\\
\hline
$6$&$0$& $0$ &$v_1v_2t^{13}+v_3v_4t^5+t^{17}$  &$272$\\
\hline
$7$&$0$& $v_1v_3t^2$ &$v_1v_2t^{13}+v_3v_4t^5+t^{17}$ & $212$\\
\hline
$8$&$0$& $v_2v_3t$ &$v_1v_2t^{13}+v_3v_4t^5+t^{17}$  &$204$\\
\hline
$9$&$0$& $0$ &$v_1v_3t^{11}+v_2v_4t^7+t^{17}$  &$272$ \\
\hline
$10$&$0$& $v_1v_2t^4$ &$v_1v_3t^{11}+v_2v_4t^7+t^{17}$  &$220$\\
\hline
$11$&$0$& $v_2v_3t$ &$v_1v_3t^{11}+v_2v_4t^7+t^{17}$  &$204$\\
\hline
$12$&$v_1v_2t$& $0$ &$v_1v_3t^{11}+v_2v_4t^7+t^{17}$  &$209$\\
\hline
$13$&$v_1v_2t$& $v_2v_3t$ &$v_1v_3t^{11}+v_2v_4t^7+t^{17}$ &$144$ \\
\hline
$14$&$0$& $0$ &$v_1v_2v_3v_4t+t^{17}$  &$272$\\
\hline
$15$&$0$& $v_1v_2t^4$ &$v_1v_2v_3v_4t+t^{17}$ &$220$ \\
\hline
$16$&$0$& $v_1v_3t^2$ &$v_1v_2v_3v_4t+t^{17}$  &$212$\\
\hline
$17$&$0$& $v_2v_3t$ &$v_1v_2v_3v_4t+t^{17}$ & $204$\\
\hline
$18$&$v_1v_2t$& $0$ &$v_1v_2v_3v_4t+t^{17}$  &$209$\\
\hline
$19$&$v_1v_2t$& $v_1v_3t^2$ &$v_1v_2v_3v_4t+t^{17}$ & $149$\\
\hline
$20$&$v_1v_2t$& $v_2v_3t$ &$v_1v_2v_3v_4t+t^{17}$ & $144$\\
\hline
\end{tabular}
}
\end{center}

For degree reasons, the $\K$-homotopy self-equivalences of $X_{\K}$ are given as  
 \begin{eqnarray*}
  {\mathcal E}(X_{\K})&=&\Bigg\{ \begin{pmatrix}
a & &&&\\
 & b&&&\\
 &&c&&\\
 &&&d&\\
 &&&&e
\end{pmatrix}\ \ | \ \ a,b,c,d,e\in \K-0\Bigg\}\\
&=&\{ (a,b,c,d,e) \mid a,b,c,d,e\in \K-0\} = (\K -0)^{\times 5}
 \end{eqnarray*}
and  the subgroups ${\mathcal E}_{\mu_n}(X_{\K})$ for  $n=1\sim 20$ are given as  
the following conditions on $a,b,c,d$ with $e=1$: 
\begin{center}
{\begin{tabular}{|c|l ||c|l|}
\hline
$1$&$ ad=bc=1$& $11$ &$ bc=d, ac=b^2c=1$ \\
\hline
$2$ &  $  a^2b=bc=1,ab=d $   & $12$ &$ab=c, a^2b=bd=1$ \\
\hline
$3$ &  $  a^2c=bc=1,ac=d$   & $13$ & $ ab=c,ab^2=d,  a^2b=ab^3=1$\\
\hline
$4$ &  $ ad=ab^2=1,ab=c  $   & $14$ & $  abcd=1$ \\
\hline
$5$ &  $ab=c,a^2b=d,  a^3b=ab^2=1 $   & $15$ & $ab=d,  a^2b^2c=1$ \\
\hline
$6$ &  $ ab=cd=1 $   & $16$ &$ac=d, a^2bc^2=1$ \\
\hline
$7$ &  $ac=d,  ab=ac^2=1 $   & $17$ &$bc=d,  ab^2c^2=1$ \\
\hline
$8$ &  $  bc=d,  ab=bc^2=1$   & $18$ & $ab=c,  a^2b^2d=1$\\
\hline
$9$ &  $ ac=bd=1$   & $19$ & $ab=c,a^2b=d, a^4b^3=1$\\
\hline
$10$ &  $ab=d,  ac=ab^2=1$   & $20$ &$ab=c,ab^2=d,  a^3b^4=1$ \\
\hline
\end{tabular}
}
\end{center}

Notice 
that when $n=5$ and $13$,
$$  {\mathcal E}_{\mu_5}(X_{\K})=\{ (a,a^{2},a^{3}, a^{4},1) ; \ a^5=1\},$$
$$   {\mathcal E}_{\mu_{13}}(X_{\K})=\{ (a,a^{3},a^{4}, a^{2},1) ; \ a^5=1\} $$
are  both the identity $(1,1,1,1,1)$
in the case (2); i.e.,
$[\mu_5]=[\mu_{13}]$ in ${\mathcal P}_{\K}(X)$.
Thus the  Hasse diagram of 
${\mathcal P}_{\K}(X)$ is given as

$$
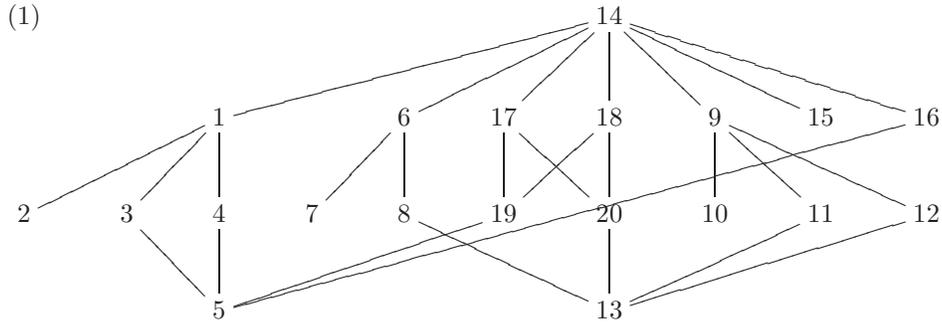
\xymatrix{
( 1)&&&&&&14\ar@{-}[lllld]\ar@{-}[lld]\ar@{-}[rrrd]\ar@{-}[ld]\ar@{-}[rrd]\ar@{-}[rd]\ar@{-}[d]&&&&&&\\
&&1\ar@{-}[lld]\ar@{-}[ld]\ar@{-}[d]&&6\ar@{-}[d]\ar@{-}[ld]&17\ar@{-}[d]\ar@{-}[rd]&
18\ar@{-}[d]\ar@{-}[ld]&9\ar@{-}[d]\ar@{-}[rd]\ar@{-}[rrd]&15&16\ar@{-}[llllllldd]&&&\\
2&3\ar@{-}[rd]&4\ar@{-}[d]&7&8\ar@{-}[rrd]&19\ar@{-}[llld]&20\ar@{-}[d]&10&11\ar@{-}[lld]&12\ar@{-}[llld]&&&\\
&&5&&&&13&&&&&&\\
}
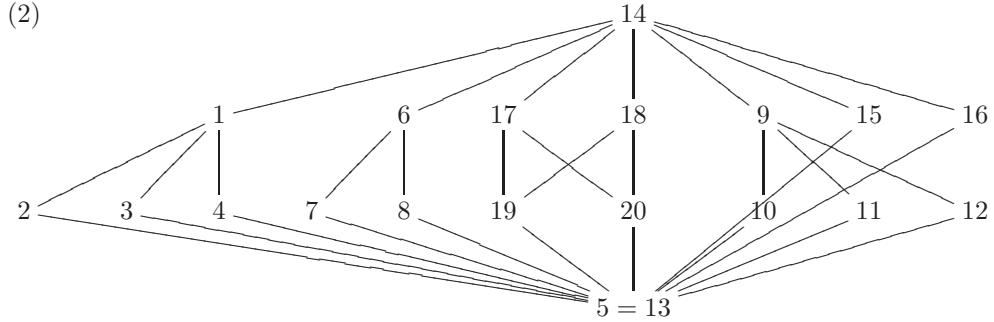
$$

$$\xymatrix{
(2)&&&&&&14\ar@{-}[lllld]\ar@{-}[lld]\ar@{-}[rrrd]\ar@{-}[ld]\ar@{-}[rrd]\ar@{-}[rd]\ar@{-}[d]&&&&&&\\
&&1\ar@{-}[lld]\ar@{-}[ld]\ar@{-}[d]&&6\ar@{-}[d]\ar@{-}[ld]&17\ar@{-}[d]\ar@{-}[rd]&
18\ar@{-}[d]\ar@{-}[ld]&9\ar@{-}[d]\ar@{-}[rd]\ar@{-}[rrd]&15\ar@{-}[lldd]&16\ar@{-}[llldd]&&&\\
2\ar@{-}[rrrrrrd]&3\ar@{-}[rrrrrd]&4\ar@{-}[rrrrd]&7\ar@{-}[rrrd]&8\ar@{-}[rrd]&19\ar@{-}[rd]&20\ar@{-}[d]&10\ar@{-}[ld]&11\ar@{-}[lld]&12\ar@{-}[llld]&&&\\
&&&&&&5=13&&&&&&\\
}$$
and  c-s-${\rm depth}_{\K}(X)= 4$ in both cases.
In particular, $(2)$ is a lattice.

\end{ex}

\begin{rem}
Let  $F_{\mu}:{\mathcal E}(p_{\mu})\to {\mathcal E}(X)$
be the restriction map between  ordinary self-homotopy equivalence groups
for the fibration $(*)$ in \S 1.
From Example \ref{ex2} $(2)$ and the integral homotopy theory \cite[\S 10]{Su},
we see that Im$(F_{\mu})$ may be at most finite
even when $X$ has the rational homotopy type of the product of odd-spheres. So we would like to pose the following question.
\end{rem}

\begin{qu}
When is  Im$(F_{\mu})$ finite?
\end{qu}

\begin{rem}
M.R.Hilali conjectures that $\dim \pi_*(Y)\otimes \Q \leq \dim H^*(Y;\Q )$
for an elliptic simply connected space $Y$ \cite{HM},\cite{HM2}.
When $Y$ is c-symplectic with $\dim \pi_{even}(Y)\otimes \Q =1$,
it is true.
Indeed, when $M(Y)=(\Q [t]\otimes \Lambda (v_1,\cdots ,v_n),d)$ ($n>1$) with $|v_i|$ odd,
$\dim \pi_*(Y)\otimes \Q=n+1<((\sum_{i=1}^n|v_i|)-1)/2=\max \{m\mid t^m\neq 0 \}=:N$.
Thus it follows from $H^*(Y;\Q )\supset \{ 1,t,t^2,\cdots ,t^N\}$.  
\end{rem}

We first  speculated  that if  $[\mu_i]<[\mu_j]$ in ${\mathcal P}_{\K}(X)$, then $\dim  H^*(Y_i;\K )\leq \dim H^*(Y_j;\K )$.
But it is not true in a general field $\K$.
Indeed, we can see in the above Example \ref{ex2} $(2)$
that  $[\mu_5]<[\mu_{20}]$ but $\dim  H^*(Y_{5};\K )=149> 144=\dim H^*(Y_{20};\K )$.
Notice that the above speculation  holds  in the case $(1)$. So the following question would be reasonable.

\begin{qu}
For a sufficiently large field $\K$,
if  $[\mu_i]<[\mu_j]$
in ${\mathcal P}_{\K}(X)$, then  is $\dim  H^*(Y_i;\K )\leq \dim H^*(Y_j;\K )$ ?
\end{qu}

\begin{ex}Recall Corollary \ref{Zn}.
For example, the following Hasse diagrams of height $3$ are contained in  those of 
${\mathcal P}_{\overline{\Q}}(X)$
with  $X=\C P^n\times S^{2n+3}$ of  c-s-${\rm depth}_{\overline{\Q}}(X)=4$
for $n=26, \ 74$ and $104$, respectively:
{\small
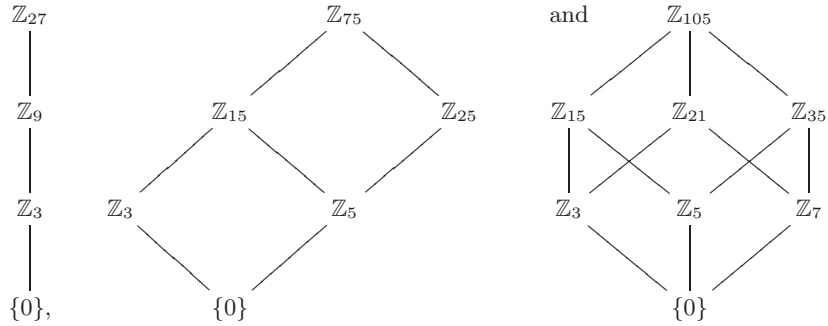
$$
\xymatrix{\Z_{27} \ar@{-}[d]\\
\Z_{9}\ar@{-}[d]\\
\Z_{3}\ar@{-}[d]\\
\{ 0\} , }\ \ \ \ \
\xymatrix{&&\Z_{75} \ar@{-}[ld]\ar@{-}[rd]&\\
 &\Z_{15}  \ar@{-}[dr]\ar@{-}[ld]&&\Z_{25}\ar@{-}[ld]\\
\Z_3\ar@{-}[rd]& &\Z_5 \ar@{-}[ld]&\\
&\{ 0\}&&}\ \ \ \ \ \ \ 
\xymatrix{\mbox{and}&\Z_{105} \ar@{-}[ld]\ar@{-}[rd] \ar@{-}[d]&\\
\Z_{15}\ar@{-}[rd]  \ar@{-}[d]&\Z_{21}  \ar@{-}[ld] \ar@{-}[rd]&\Z_{35}\ar@{-}[ld] \ar@{-}[d]\\
\Z_3  \ar@{-}[dr]&\Z_5 \ar@{-}[d]&\Z_7 \ar@{-}[dl]\\
&\{ 0\}&
}$$
}
\end{ex}

\begin{ex}(c-s-${\rm depth}_{\K}(X)=5$).  For the exceptional simple Lie group $E_7$, the rational type is known as (see \cite{M}) 
$$ (3,11,15,19,23,27,35)$$
namely, 
$$ E_7\simeq_{\Q}  S^3 \times S^{11} \times S^{15} \times S^{19}\times S^{23}\times S^{27}\times S^{35}.$$

Then, for  $|v_1|=3,\ |v_2|=11,\ |v_3|=15,\ |v_4|=19,\ |v_5|=23,\ |v_6|=27,\ |v_7|=35$,
$$  {\mathcal E}(X_{\K})=Aut(\Lambda (v_1,v_2,v_3,v_4,v_5,v_6,v_7)_{\K},0)
\cong \Big\{ {\rm diag}(a,b,c,d,e,f,g)\mid a,b,c,d,e,f,g\in \K -0\Big\}$$
$$=\Big\{ (a,b,c,d,e,f,g)\mid a,b,c,d,e,f,g\in \K -0\Big\}=(\K-0)^{\times 7}$$  for  degree reasons.
 There are the following 20-types of $\K$-c-symplectic models, i.e.;
$$(\K [t]\otimes \Lambda (v_1,v_2,v_3,v_4,v_5,v_6,v_7)_{\K},D)$$
 with  
the differentials  given as  
$$Dv_1=Dv_2=0,$$
$$Dv_7=v_1v_6t^3+v_2v_5t+v_3v_4t+t^{18}\ \ \mbox{ and}$$
\begin{center}{\begin{tabular}{|c|c|c|c|c|c|}
\hline
$\mu$  & $Dv_3$ & $Dv_4$ &$Dv_5$ & $Dv_6$ &  $  {\mathcal E}_{\mu}(X_{\K})$ \\
\hline
$1$     & $0$& $0$         &$0$ & $0$ & $(a,b,c,c^{-1},b^{-1},a^{-1},1)$\\
\hline
$2$    & $0$ & $v_1v_3t$         &$0$ & $0$ & $(c^{-2},b,c,c^{-1},b^{-1},c^{2},1)$\\
\hline
$3$    & $0$ & $0$         &$v_1v_2t^5$ & $0$ & $(b^{-2},b,c,c^{-1},b^{-1},b^{2},1)$\\
\hline
$4$     & $0$& $0$         &$0$ & $v_1v_2t^7$ & $(a,a^{-2},c,c^{-1},a^{2},a^{-1},1)$\\
\hline
$5$    & $0$ & $0$         &$0$ & $v_1v_3t^5$ & $(a,b,a^{-2},a^{2},b^{-1},a^{-1},1)$\\
\hline
$6$     & $0$& $0$         &$0$ & $v_1v_4t^3$ & $(a,b,a^2,a^{-2},b^{-1},a^{-1},1)$\\
\hline
$7$    & $0$ & $0$         &$0$ & $v_1v_5t$ & $(a,a^2,c,c^{-1},a^{-2},a^{-1},1)$\\
\hline
$8$     & $0$& $v_1v_3t$         &$v_1v_2t^5$ & $0$ & $(b^{-2},b,\pm b,\pm b^{-1},b^{-1},b^{2},1)$\\
\hline
$9$    & $0$ & $v_1v_3t$         &$0$ & $v_1v_2t^7$ & $(c^{-2}, c^{4}, c,c^{-1}, c^{-4}, c^2,1)$\\
\hline
$10$     & $0$& $v_1v_3t$         &$0$ & $v_1v_4t^3$ & $(c^{-2},b,c,c^{-1},b^{-1},c^{2},1):\ c^5=1$\\
\hline
$11$    & $0$ & $v_1v_3t$         &$0$ & $v_1v_5t$ & $(c^{-2}, c^{-4}, c,c^{-1}, c^{4}, c^2,1)$\\
\hline
$12$    & $0$ & $0$         &$v_1v_2t^5$ & $v_1v_3t^5$ & $(b^{-2},b,b^4, b^{-4}, b^{-1},b^2,1)$\\
\hline
$13$     & $0$& $0$         &$v_1v_2t^5$ & $v_1v_4t^3$ & $(b^{-2},b,b^{-4}, b^{4}, b^{-1},b^2,1)$\\
\hline
$14$    & $0$ & $0$         &$v_1v_2t^5$ & $v_1v_5t$ & $(b^{-2},b,c,c^{-1},b^{-1},b^{2},1):\ b^5=1$\\
\hline
$15$     & $0$& $v_1v_3t$         &$v_1v_2t^5$ & $v_1v_4t^3$ & $(b^{-2},b,\pm b,\pm b^{-1},b^{-1},b^{2},1); \ b^5=\pm 1$\\
\hline
$16$     & $0$& $v_1v_3t$         &$v_1v_2t^5$ & $v_1v_5t$ & $(b^{-2},b,\pm b,\pm b^{-1},b^{-1},b^{2},1); \ b^5= 1$\\
\hline
$17$     & $0$& $0$         &$v_1v_3t^3$ & $v_2v_3t$ &$(b^{-1}c^{-1},b,c,c^{-1},b^{-1},bc,1)$\\
\hline
$18$     & $0$& $v_1v_2t^3$         &$v_1v_3t^3$ & $0$ &$(a,b,a^{-1}b^{-1}, ab,b^{-1},a^{-1},1)$\\
\hline
$(18)$     & $0$& $-v_1v_2t^3$         &$0$ & $v_2v_3t$ &$(a,b,a^{-1}b^{-1}, ab,b^{-1},a^{-1},1)$ \\
\hline
$19$     & $0$& $v_1v_2t^3$         &$2v_1v_3t^3$ & $v_2v_3t$ & $(a,b,ab, a^{-1}b^{-1},b^{-1},a^{-1},1)$\\
\hline
$(19)$     & $v_1v_2t$& $0$         &$v_1v_4t$ & $0$ & $(a,b,ab, a^{-1}b^{-1},b^{-1},a^{-1},1)$\\
\hline
$20$     & $v_1v_2t$& $-v_1v_2t^3$         &$-v_1v_4t$ & $v_2v_3t$ &  $(a,\pm a^{-1},\pm 1, \pm 1,\pm a,a^{-1},1)$\\
\hline

\end{tabular}
}
\end{center}
for   $a,b,c\in \K-0$.   
Then, for any $\K$, the 
Hasse diagram is 

$$
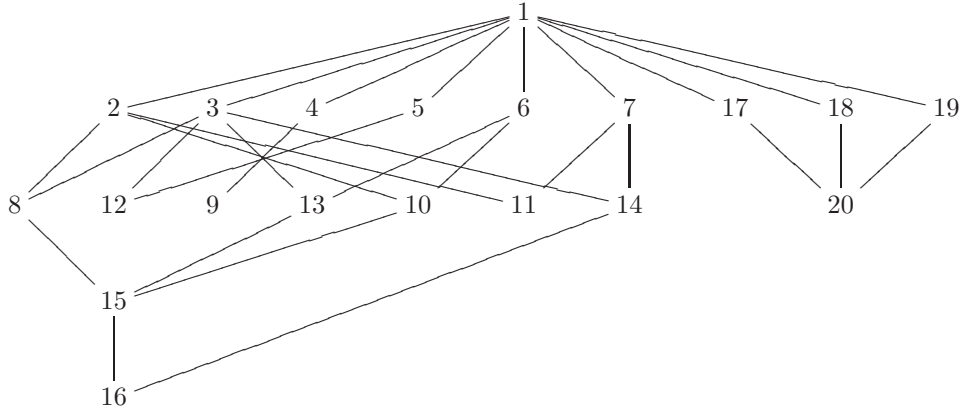
\xymatrix{
&&&&&1\ar@{-}[lllld]\ar@{-}[rd]\ar@{-}[llld]\ar@{-}[d]\ar@{-}[ld]\ar@{-}[lld]\ar@{-}[rrd]\ar@{-}[rrrd]\ar@{-}[rrrrd]&&&&&\\
&2\ar@{-}[drrr]\ar@{-}[ld]\ar@{-}[rrrrd]&3\ar@{-}[ld]\ar@{-}[lld]\ar@{-}[rd]\ar@{-}[rrrrd]&
4\ar@{-}[ld]&5\ar@{-}[llld]&6\ar@{-}[ld]\ar@{-}[lld]&7\ar@{-}[ld]\ar@{-}[d]&17\ar@{-}[dr]&18\ar@{-}[d]&19\ar@{-}[ld]&\\
8\ar@{-}[dr]&12&9&13\ar@{-}[lld]&10\ar@{-}[llld]&11&14\ar@{-}[ddlllll]&&20&&&\\
&15\ar@{-}[d]&&&&&&&&&&\\
&16&&&&&&&&&&\\
}$$

\end{ex}

    \begin{cor}  The poset structure of ${\mathcal P}_{\K}(E_7)$ does not depend on $\K$.
Moreover
   c-s-${\rm depth}_{\K}(E_7)=5$ for any field $\K$.
 \end{cor}

\begin{ex}\label{ex5.11}
For the n-dimensional symplectic group $Sp(n)$,
the rational type is given as 
$(3,7,11,\cdots ,4n-1)$ (see \cite{M}),
namely,
$$Sp(n)\simeq_{\Q} S^3 \times S^7 \times \cdots \times S^{4k-1} \times \cdots \times S^{4n-1}.$$
It is pre-c-symplectic if $n$ is odd \cite{SY}.
Let

$$M(Sp(n)_{\K})=(\Lambda (v_1,v_2,v_3,\cdots ,v_n)_{\K},0) \quad \text{with} \quad  |v_i|=4i-1.$$
Then, for  degree reasons \cite{SY}, for the differential $D$ over a sufficiently large field  $\K$,
$$Dv_n=v_1v_{n-1}t+ v_2v_{n-2}t+ v_3v_{n-3}t+ \cdots  +v_{(n-1)/2}v_{(n+1)/2}t+ t^{2n}$$
is uniquely defined to be c-symplectic. Thus by seeing  $Dv_1, \cdots ,\ Dv_{n-1}$ we have 
the following:

(1) c-s-${\rm depth}_{\K}(Sp(1))=$c-s-${\rm depth}_{\K}(Sp(3))=1$.

(2)  c-s-${\rm depth}_{\K}(Sp(5))=3$ from the sequence $[\mu_1]>[\mu_2]>[\mu_3]$
for 
$$M({Y_i}_{\K})=(\K [t]\otimes \Lambda (v_1,v_2,v_3,v_4 ,v_5)_{\K},D)$$ 
with $Dv_5=v_1v_4t+v_2v_3t+t^5$ and 
 $$M({Y_1}_{\K})\ :\ \ Dv_1=Dv_2=0, \ Dv_3=0,\ Dv_4=0 \qquad \qquad $$
$$M({Y_2}_{\K})\ :\ \ Dv_1=Dv_2=0, \ Dv_3=v_1v_2t,\ Dv_4=0 \qquad $$
$$M({Y_3}_{\K})\ :\ \ Dv_1=Dv_2=0, \ Dv_3=v_1v_2t,\ Dv_4=v_1v_3t.$$
Then we have 
 \begin{eqnarray*}
 {\mathcal E}_{\mu_1}(X_{\K})&=&\{ (a,b,b^{-1},a^{-1},1)\},\\
 {\mathcal E}_{\mu_2}(X_{\K})&=&\{ (b^{-2},b,b^{-1},b^{2},1)\},\\
 {\mathcal E}_{\mu_3}(X_{\K})&=&\{ (b^{-2},b,b^{-1},b^{2},1);\ b^5=1\}
 \end{eqnarray*}
for   $a,b\in \K-0$
and we can directly check that the  sequence $ {\mathcal E}_{\mu_1}(X_{\K}) \supsetneq {\mathcal E}_{\mu_2}(X_{\K}) \supsetneq
 {\mathcal E}_{\mu_3}(X_{\K})$ is maximal.\\

(3)  c-s-${\rm depth}_{\K}(Sp(7))=4$ since $M({Y_i}_{\K})=(\K [t]\otimes \Lambda (v_1,v_2,v_3,v_4 ,v_5,v_6,v_7)_{\K},D)$ with 
$Dv_1=Dv_2=Dv_3=0$,
$Dv_{7}=v_1v_{6}t+ v_2v_{5}t+ v_3v_{4}t + t^{14}$
and 
\begin{center}{\begin{tabular}{|c|c|c|c|c|}
\hline
&$Dv_4$& $Dv_5$ & $Dv_6$ &   $  {\mathcal E}_{\mu_i}(X_{\K})$\\
\hline
$\mu_1$& $0$& $0$ & $0$ &   $\{ (a,b,c)\mid a,b,c\in \K-0\}$ \\
\hline
$\mu_2$&  $v_3v_1t$& $0$ & $0$ &  $a=c^{-2}$\\
\hline
$\mu_3$&  $v_3v_1t$& $v_2v_3t$ & $0$ &  $ a=c^{-2},c=b^{-2} $\\
\hline
$\mu_4$&  $v_3v_1t  $& $v_2v_3t$ & $v_1v_2t^7$ & $ a=c^{-2},c=b^{-2},b=a^{-2}$\\
\hline
\end{tabular}\\
}
\end{center}
for  $(a,b,c):={\rm diag}(a,b,c,c^{-1},b^{-1},a^{-1},1)$
and we can directly check that the  sequence $ {\mathcal E}_{\mu_1}(X_{\K}) \supsetneq {\mathcal E}_{\mu_2}(X_{\K}) \supsetneq
 {\mathcal E}_{\mu_3}(X_{\K})\supsetneq
 {\mathcal E}_{\mu_4}(X_{\K})$ is maximal.\\

(4)  c-s-${\rm depth}_{\K}(Sp(9))\geq 5$. Indeed, $M({Y_i}_{\K})=(\K [t]\otimes \Lambda (v_1,v_2,v_3,v_4 ,v_5,v_6,v_7,v_8,v_9)_{\K},D)$
with
$Dv_1=Dv_2=\cdots =Dv_4=0$, 
$Dv_{9}=v_1v_{8}t+ v_2v_{7}t+ v_3v_{6}t+ v_4v_5t + t^{18}$ and
\begin{center}{\begin{tabular}{|c|c|c|c|c|c|}
\hline
&$Dv_5$& $Dv_6$ & $Dv_7$ & $Dv_8$ &  $  {\mathcal E}_{\mu_i}(X_{\K})$\\
\hline
$\mu_1$& $0$& $0$ & $0$ & $0$ &  $\{ (a,b,c,d)\mid a,b,c,d\in \K-0\}$ \\
\hline
$\mu_2$&  $v_4v_1t$& $0$ & $0$ &  $0$ & $a=d^{-2}$\\
\hline
$\mu_3$&  $v_4v_1t$& $v_3v_1t^5$ & $0$ &  $0$ & $ a=d^{-2}=c^{-2} $\\
\hline
$\mu_4$&  $v_4v_1t  $& $v_3v_1t^5$ & $v_2v_4t^3$ & $0$ & $ a=d^{-2}=c^{-2},d=b^{-2} $\\
\hline
$\mu_5$&  $ v_4v_1t $& $v_3v_1t^5$ & $v_2v_4t^3$ & $v_1v_2t^{13}$ & $ a=d^{-2}=c^{-2},d=b^{-2} ,b=a^{-2}$\\
\hline
\end{tabular}\\
}
\end{center}
for  $(a,b,c,d):={\rm diag}(a,b,c,d,d^{-1},c^{-1},b^{-1},a^{-1},1)$.

Thus we have $ {\mathcal E}_{\mu_1}(X_{\K}) \supsetneq {\mathcal E}_{\mu_2}(X_{\K}) \supsetneq
\cdots \supsetneq {\mathcal E}_{\mu_5}(X_{\K})$; i.e.,
$$ [\mu_1]> [\mu_2]> [\mu_3]> [\mu_4]> [\mu_5] ,$$
which 
implies   c-s-${\rm depth}_{\K}(Sp(9))\geq 5$.\\

(5)  c-s-${\rm depth}_{\K}(Sp(11))\geq 6$.
Indeed, 
$$ M({Y_i}_{\K})=(\K [t]\otimes \Lambda (v_1,v_2,v_3,v_4 ,v_5,v_6,v_7,v_8,v_9,v_{10},v_{11})_{\K},D) \ \ {\rm with}$$
$$Dv_1=Dv_2=\cdots =Dv_5=0,$$
$$Dv_{11}=v_1v_{10}t+ v_2v_{9}t+ v_3v_{8}t+ v_4v_7t+v_5v_6t + t^{22}\mbox{ \ and}$$

\begin{center}{\begin{tabular}{|c|c|c|c|c|c|c|}
\hline
& $Dv_6$ & $Dv_7$ & $Dv_8$ & $Dv_9$ & $Dv_{10}$ & $  {\mathcal E}_{\mu_i}(X_{\K})$\\
\hline
$\mu_1$&  $0$ & $0$ & $0$ & $0$ & $0$ & $\{ (a,b,c,d,e)\mid a,b,c,d,e\in \K-0\}$ \\
\hline
$\mu_2$&  $v_5v_1t$ & $0$ & $0$ & $0$ & $0$ & $a=e^{-2}$\\
\hline
$\mu_3$&  $v_5v_1t$ & $v_4v_1t^5$ & $0$ & $0$ & $0$ & $ a=e^{-2}=d^{-2} $\\
\hline
$\mu_4$&  $v_5v_1t$ & $v_4v_1t^5$ & $v_3v_2t^7$ & $0$ & $0$ & $ a=e^{-2}=d^{-2},b=c^{-2} $\\
\hline
$\mu_5$&  $v_5v_1t$ & $v_4v_1t^5$ & $v_3v_2t^7$ & $v_2v_4t^2$ & $0$ & $ a=e^{-2}=d^{-2},b=c^{-2},d=b^{-2} $
\\
\hline
$\mu_6$&  $v_5v_1t$ & $v_4v_1t^5$ & $v_3v_2t^7$ & $v_2v_4t^2$ & $v_1v_3t^{15}$ & $a=e^{-2}=d^{-2},b=c^{-2},d=b^{-2},c=a^{-2} $
\\
\hline
\end{tabular}\\
}
\end{center}
for  $(a,b,c,d,e):={\rm diag}(a,b,c,d,e,e^{-1},d^{-1},c^{-1},b^{-1},a^{-1},1)$.

Thus we have $ {\mathcal E}_{\mu_1}(X_{\K}) \supsetneq {\mathcal E}_{\mu_2}(X_{\K}) \supsetneq
\cdots \supsetneq {\mathcal E}_{\mu_6}(X_{\K})$; i.e.,
$$ [\mu_1]> [\mu_2]> [\mu_3]> [\mu_4]> [\mu_5]> [\mu_6] ,$$
which 
implies  c-s-${\rm depth}_{\K}(Sp(11))\geq 6$.
\end{ex}

 In general, for $Sp(n)$, 
by setting 

$$ M(Y_{(n+1)/2})=(\Q [t]\otimes \Lambda (v_1,v_2,v_3,\cdots ,v_{n}),D)$$
with
$$Dv_1=Dv_2=\cdots =Dv_{(n-1)/2}=0,$$ 
$$Dv_i= \epsilon_i v_jv_{k}t^{2(i-j-k)+1}\ \ \ (\epsilon_i=0\mbox{ or }1)$$
where $i+j=n$ for $(n+1)/2\leq i<n$ and certain $k\in \{ 1, 2,\cdots ,(n-1)/2\}$ 
with $i\geq j+k$,
we obtain the following
\begin{thm}\label{sp}
For any odd integer $n>3$, 
  $$\mbox{c-s-}{\rm depth}_{\overline{\Q}}(Sp(n))\geq \frac{n+1}{2}.$$
\end{thm}

\noindent
{\bf Acknowledgements}. 
 The authors are grateful to the  referee for his/her many  valuable comments and  suggestions.




\end{document}